\documentclass[12pt,a4paper,twoside]{article}
\usepackage{amsmath,amssymb,amscd}
\numberwithin{equation}{section} 

\newtheorem{theorem}{Theorem}[section]
\newtheorem{lemma}[theorem]{Lemma}
\newtheorem{definition}[theorem]{Definition}

\newenvironment{proof}{\paragraph{Proof.}}{\hfill $\square$\\}
\newenvironment{proof*}{\paragraph{Proof.}}{}

\begin{document}

\title{Weak Solutions of a Hyperbolic-Type Partial Dynamic Equation in Banach Spaces}
\author{Ahmet Yantir$^\dag$ and Duygu Soyo\u{g}lu$^\ddag$\\[3mm]
\small (\dag) Department of Mathematics, Ya\c{s}ar University, \.{I}zmir, Turkey\\
\small e-mail: {\tt ahmet.yantir@yasar.edu.tr}\\[2mm]
\small phone: +903234115107 \ fax: +902323745474\\
\small (\ddag) Department of Mathematics, \.{I}zmir University of Economics, \.{I}zmir, Turkey\\
\small e-mail: {\tt duygusoyoglu@gmail.com}}
\date{}
\maketitle

\begin{abstract}
In this article, we prove an existence theorem regarding the weak
solutions to the hyperbolic-type partial dynamic equation
\begin{equation*}\begin{array}{l}
 z^{\Gamma\Delta}(x,y)=f(x, y, z(x, y)), \\
 z(x, 0)=0, \ \ \ z(0, y)=0 \\
 \end{array}, \ \ x\in\mathbb{T}_1, \ \ y\in \mathbb{T}_2\end{equation*}
in Banach spaces. For this purpose, by generalizing the definitions and results of Cicho\'{n} \emph{et.al.} we develop weak partial derivatives, double integrability and the mean value
results for double integrals on time scales.  DeBlasi measure of weak noncompactness  and Kubiaczyk's fixed point theorem for the weakly sequentially continuous mappings  are the  essential tools to prove the main
result.


\end{abstract}
\section{Introduction}
The \emph{time scale} which unifies the discrete and continuous
analysis was initiated by Hilger \cite{Hilger1}. Henceforth, the equations
which can be described by continuous and discrete models are
unified as "dynamic equations". Nevertheless, the theory of dynamic
equations does not provide only a unification of continuous and discrete
models. It also gives an opportunity to study some difference schemes
based on variable step-size such as $q$-difference (quantum) models under the frame of dynamic equations.
The landmark studies in the theory of dynamic equations are
collected in the books by Bohner and Peterson \cite{Book1, Book2}.

Since the difference and differential equations are also studied
in infinite dimensional Banach spaces
\cite{ravi2,cellina,cichon3,cichon2,cichon,cramer,davidovski,gonzalez,kuma,kubi2,kubiaczyk,szufla}, it is reasonable to study dynamic equations in Banach
spaces. The pioneering work on dynamic equations in Banach spaces
is by Hilger \cite{Hilger2}. Nevertheless this area is not sufficiently developed.  Recently, Cicho\'{n} et. al. \cite{cksy} study the existence of
weak solutions of Cauchy dynamic problem. After this work, there
have been some research activities in the theory of dynamic
equations in Banach spaces \cite{peano, cksy2, polonya1}.

On the other hand, the bi-variety calculus on time scales dates back to the landmark articles of Bohner and Guseinov \cite{pd, mi}. Authors study the partial differentiation and multiple integration on time scales respectively.
Jackson \cite{jackson} and  Ahlbrandt and Morian
\cite{ahlbrandt} employ these background for studying some specific kinds of partial
dynamic equations on $\mathbb{R}$. However, there is no result for the partial dynamic equations in
Banach spaces.

The hyperbolic Goursat problem $$u_{xy}=f(x, y, u, u_x, u_y), \ \ u(x,0)=u(0,y)=0, \ \ (x,y)\in V$$ where $V$ is a rectangle containing $(0, 0)$,  has been studied by many authors for years. Picard proved that when $f(x,y,z_1,z_2,z_3)$ is Lipschitz continuous in the $z-$ variable, then the solution exists and unique \cite{darboux, kamke}. The existence of solutions when $f$ is independent from $z_2$ and $z_3$ was proved by Montel \cite{montel}. Then the sharper results followed by weakening the conditions on $f$ (see \cite{leehey, hart, alex-orli, persson1, persson2}). For an application of a hyperbolic partial differential equations in stochastic process, see \cite{stehlik}.

Motivated by the above studies and the lack of the results for nonlinear partial dynamic equations, in this article, we concentrate on the hyperbolic type partial
dynamic problem
\begin{equation}\label{pde}\begin{array}{l}
 z^{\Gamma_w\Delta_w}(x,y)=f(x, y, z(x, y)), \\
 z(x, 0)=0, \ \ \ z(0, y)=0 \\
 \end{array}, \ \ x\in\mathbb{T}_1, \ \ y\in \mathbb{T}_2\end{equation} in Banach spaces. Here the time scales $\mathbb{T}_1$ and $\mathbb{T}_2$ both include $0$ and the
differential operators $\Gamma_w$ and $\Delta_w$ are weak partial derivative operators with respect to the variables $x$ and $y$ respectively.

We assume that $f$ is Banach space-valued, weakly-weakly sequentially continuous function. We also assume some regularity conditions expressed in terms of DeBlasi measure of weak noncompactness  \cite{blasi} on $f$. We define a weakly sequentially continuous integral operator associated to an integral equation which is equivalent to \eqref{pde}. The existence of a fixed point of such operator is verified by using the fixed point theorem for weakly sequentially continuous mappings given by Kubiaczyk \cite{kubi2}.

\section{Preliminaries and Notations}\label{prelim}
The time scale calculus (and weak calculus) for the Banach space valued functions is created by Cicho\'{n} \emph{et.al.} \cite{cksy,cksy2}. Authors generalize the definitions of Hilger \cite{Hilger1}. On the other hand, the multi-variable time scale calculus is created by Ahlbrandt and Morian \cite{ahlbrandt} and Jackson \cite{jackson}. In this section, we construct the definitions of weak partial derivatives and the weak double integral of a Banach space
valued function defined on $\mathbb{T}=\mathbb{T}_1\times \mathbb{T}_2$.  Also the mean value result of Cicho\'{n}
\emph{et.al.} (see Thm 2.11 of \cite{cksy}) is generalized for the multivariable case.

Before we state the preliminary definitions, we remark the readers about the notations.  Throughout this article, if a function of two variables $f: \mathbb{T}_1\times \mathbb{T}_2\to E$ is considered, by $f^\Gamma(s,t)$, we mean the forward $\Gamma$-derivative with respect to the first variable $s\in \mathbb{T}_1$. Similarly $f^\Delta(s,t)$ stands for the forward $\Delta$-derivative with respect to the second variable $t\in \mathbb{T}_2$. For a function of single variable $f: \mathbb{T}\to E$, the ordinary  notation $f^\Delta(t)$ is used. The similar considerations are also valid for the integrals.

We refer to \cite{cksy} for the weak calculus of functions of single variable defined on a time scale. We only state the core definitions to clarify the weak calculus of functions of several variables defined on product time scale.

Let $(E, ||\cdot||)$ be a Banach space with the supremum norm and $E^*$ be its dual space.

\begin{definition}\label{weakderivative}\emph{\cite{cksy}}
Let $f : \mathbb{T} \rightarrow E$. Then we say that $f$ is
$\Delta$-weak differentiable at $t\in\mathbb{T}$ if there exists
an element $F(t) \in E$ such that for each $x^{*}\in  E^{*}$ the
real valued function $x^{*}f$ is $\Delta$-differentiable at $t$
and ($x^{*}f)^{\Delta} (t) = (x^{*}F)(t)$. Such a function $F$ is
called $\Delta$-weak derivative of $f$ and denoted by $f^{\Delta_w}$.
\end{definition}
\begin{definition}\emph{\cite{jackson}}\emph{(}Partial Differentiability\emph{)}\label{partial}
Let $f:\mathbb{T}=\mathbb{T}_1\times\mathbb{T}_2\rightarrow \mathbb{R}$ be a function and let $(s, t) \in \mathbb{T}^{k}.$ We define $f^{\Gamma}(s,t)$ to be the number (provided that it exists) with the property that given any $\varepsilon > 0$, there exists a neighborhood $N$ of $s$, with $N=(s-\delta,s+\delta) \cap \mathbb{T}_{1}$ for $\delta>0$ such that
$$\left| [f(\sigma(s),t)-f(u,t)]- f^{\Gamma}(s,t)[\sigma(s)-u]\right| \leq \varepsilon \left| \sigma(s)-u \right|$$ for all $u \in N.$ $f^{\Gamma}(s,t)$ is called  the partial delta derivative of $f$ with respect to the variable $s$.

Similarly we define $f^{\Delta}(s,t)$ to be the number (provided that it exists) with the property that given any $\varepsilon > 0$, there exists a neighborhood $N$ of $t$, with $N=(t-\delta, t+\delta) \cap \mathbb{T}_{2}$ for $\delta>0$ such that
$$\left| [f(s,\sigma(t))-f(s,u)]- f^{\Delta}(s,t)[\sigma(t)-u]\right| \leq \varepsilon \left| \sigma(t)-u \right|$$ for all $u \in N.$ $f^{\Delta}(s,t)$ is called  the partial delta derivative of $f$ with respect to the variable $t$.
\end{definition}
Since we have the definitions of weak $\Delta$-derivative and the partial derivatives on time scales, it is reasonable to combine these definitions to construct the definition of weak partial derivative of a Banach space valued function.
\begin{definition}\label{weakpartialderivative}
Let $f : \mathbb{T}_1\times \mathbb{T}_2 \rightarrow E$. Then we say that $f$ is
$\Gamma$-weak partial differentiable at $(s,t)\in\mathbb{T}$ if there exists
an element $F(s, t) \in E$ such that for each $z^{*}\in  E^{*}$ the
real valued function $z^{*}f$ is $\Gamma$ partial differentiable at $(s,t)$
and ($z^{*}f)^{\Gamma} (s,t) = (z^{*}F)(s,t)$. Such a function $F$ is
called $\Gamma$-weak partial derivative of $f$ and denoted by $f^{\Gamma_
w}$.

Similarly, $f$ is said to be $\Delta$-weak partial differentiable at $(s,t)\in\mathbb{T}$ if there exists
an element $F(s, t) \in E$ such that for each $z^{*}\in  E^{*}$ the
real valued function $z^{*}f$ is $\Delta$ partial differentiable at $(s,t)$
and ($z^{*}f)^{\Delta} (s,t) = (z^{*}F)(s,t)$. Such a function $F$ is
called $\Delta$-weak partial derivative of $f$ and denoted by $f^{\Delta_
w}$.
\end{definition}
\begin{definition}\label{weakcauchy}
If $F^{\Gamma_w}(s,t) = f(s,t)$ for all $(s,t)$, then we define $\Gamma$-weak
Cauchy integral by
$$(C_w)\int^{s}_{a}f(\tau,t)\Gamma\tau=F(s, t)-F(a,t).$$
\end{definition}
The Riemann, Cauchy-Riemann, Borel and Lebesgue integrals on time scales for the Banach space-valued functions are defined by Aulbach \emph{et. al.} \cite{aulbach}. Since the weak Cauchy integral is defined by means of weak anti-derivativatives, the space of weak integrable functions is too restricted. Therefore it is conceivable to define the weak Riemann integral.
\begin{definition}\label{finer}\emph{\cite{cksy}}
Let $P = \{a_{0},a_{1},\cdots,a_{n}\}$ be a partition of the
interval $[a,b]$. P is called $finer$ than $\delta>0$ either $\mu_{\mathbb{T}}([a_{i-1},a_{i}])\leq \delta$ or
 $\mu_{\mathbb{T}}([a_{i-1},a_{i}])> \delta$ only if $a_{i}=\sigma(a_{i-1})$, where $\mu_{\mathbb{T}}$ denotes the time scale measure.
\end{definition}
\begin{definition}\emph{(}Riemann Double Integrability\emph{)}\label{riemdouble}
A Banach space valued-function $f: [a, b]\times [c, d]\to E$ is
called weak Riemann double integrable if there exists $I\in E$
such that for any $\varepsilon>0$ there exists a positive number
$\delta$ with the following property:  For any partition
$P_1=\{a_0, a_1, \cdots, a_n\}$ of $[a, b]$ and
$P_2=\{c_0, c_1, \cdots, c_n\}$ of $[c, d]$ which are
finer than $\delta$ and the set of points $s_j\in [a_{j-1}, a_j)$
and $t_j\in[c_{j-1},c_j)$ for $j=1, 2, \cdots, n$ one has
$$\left|z^{*}(I)-\sum_{j=1}^{n}z^{*}(f(s_{j},t_{j}))\mu_{\mathbb{T}}([a_{j-1}, a_j)\times[c_{j-1}c_j))\right|\leq \ \varepsilon,   \text{ for all } z^{*}\in E^{*}.$$
The uniquely determined function $I$ is called weak Riemann double integral $f$ and
denoted by
$$I=(\mathcal{R}_w)\int\int_{[a,b]\times [c,d]}f(s,t)\Delta t\Gamma s.$$
\end{definition}
Using Theorem 4.3 of Guseinov \cite{Guseinov} and regarding the definition of weak Cauchy and Riemann integrals, it can be remarked that every Riemann weak integrable function is Cauchy weak integrable and therefore these two integrals coincide.

The measure of weak noncompactness which is developed by DeBlasi \cite{blasi} is the fundamental tool in our main result. The regualrity conditions on the nonlinear term $f$ is expressed in terms of measure of weak noncompactness. Let $A$ be a bounded nonempty
subset of $E$. The measure of weak noncompactness $\beta (A)$ is
defined by
$$\beta(A)=\inf\{t>0: \text{there exists} \ C\in K^\omega \ \text{such that} \ A\subset C+t B_1\}$$
where $K^\omega$ is the set of weakly compact subsets of $E$ and
$B_1$ is the unit ball in $E$.

We make use of the following properties of the measure of weak
noncompactness $\beta$. For bounded nonempty subsets $A$ and $B$
of $E$,
\begin{itemize}
\item[(1)] If $A\subset B$ then $\beta (A)\le \beta (B)$,
\item[(2)] $\beta (A)=\beta (\bar {A^w})$, where $\bar {A^w}$ denotes
the weak closure of $A$,
\item[(3)] $\beta (A)=0$ if and only if $A$ is relatively weakly
    compact,
\item[(4)] $\beta (A\cup B)=\max \left\{ {\beta (A),\beta (B)}
\right\}$,
\item[(5)] $\beta (\lambda A)=\vert \lambda \vert \beta (A)$ ($\lambda
\in R)$,
\item[(6)] $\beta (A+B)\le \beta (A)+\beta (B)$,
\item[(7)] $\beta (\overline{\text{conv}} (A))=\beta (\text{conv}
(A))=\beta (A)$, where $\text{conv}(A)$
    denotes the convex hull of $A$.
\end{itemize}
If $\beta$ is an arbitrary set function satisfying the above properties \emph{i.e.}, if $\beta$ is an axiomatic measure of weak noncompactness,
then the following lemma is true. \begin{lemma}\label{E_1} If
$||E_1||=\sup\{||x||: x\in E_1\}<1$ then $$\beta(E_1+E_2)\leq
\beta(E_2)+||E_1||\beta(K(E_2, 1)),$$ where $K(E_2,
1)=\{x:d(E_2,x)\leq 1\}$.
\end{lemma}
The generalization of Ambrosetti Lemma for
$C(\mathbb{T}_1\times\mathbb{T}_2, E)$ is as follows:
\begin{lemma}\label{ambrosetti}
Let $H \subset C(\mathbb{T}_1\times\mathbb{T}_2, E)$ be a family
of strongly equicontinuous functions. Let $H(x, y) = \{h(x, y) \in
E, h \in H\}$, for $(x, y) \in \mathbb{T}_1\times\mathbb{T}_2$.
Then
$$\beta(H(\mathbb{T}_1\times\mathbb{T}_2))=\sup_{(x, y)\in \mathbb{T}_1\times\mathbb{T}_2}\beta(H(x, y)),$$
and the function $(x, y)\mapsto \beta(H(x, y))$ is continuous on
$\mathbb{T}_1\times\mathbb{T}_2$.
\end{lemma}
\begin{proof}
The proof directly follows by generalizing the proof of Lemma 2.9 of \cite{cksy}. \end{proof}
\begin{theorem}\label{mvtdouble}(Mean Value Theorem for Double Integrals)
If the function $\phi:\mathbb{T}_{1}\times
\mathbb{T}_{2}\rightarrow E$ is $\Delta$- and $\Gamma$-weak
integrable, then
\[ \iint_{\Omega}\phi(s,t)\Delta t\Gamma s\in \mu_{\mathbb{T}}(\Omega)\cdot\overline{conv}\ \phi(\Omega)\]
where $\Omega$ is an arbitrary subset of $\mathbb{T}_{1}\times
\mathbb{T}_{2}$.
\end{theorem}
\begin{proof}
Let $\displaystyle{\iint_{\Omega}\phi(s,t)\Delta t\Gamma s=w}$ and
$\mu_{\mathbb{T}}(R)\cdot\overline{conv}\ \phi(\Omega)=W$. Suppose to the
contrary, that $w\notin W$. By separation theorem for the convex
sets there exists  $z^{*}\in E^{*}$ such that
$$\sup_{\varphi\in W}z^{*}(\varphi)=\alpha<z^{*}(w ).$$ However
$$z^{*}(w)=z^{*}\left((C_w)\iint_{\Omega}\phi(s,t)\Delta t\Gamma s\right)=\iint_{\Omega} z^{*}(\phi(s,t))\Delta t\Gamma s. $$
Moreover, for $(s,t)\in \Omega$,  we have $\phi(s,t)\in \phi(\Omega)$ and
therefore
\[\mu_{\mathbb{T}}(\Omega)\cdot \phi(s,t)\subseteq \mu_{\mathbb{T}}(\Omega)\cdot \overline{conv}\ \phi(\Omega)=W, \ \ \text{i.e.} \ \
\phi(s,t)\subseteq\frac{W}{\mu_{\mathbb{T}}(\Omega)}.\] Hence
\[z^{*}(\phi(s,t))\leq z^{*}\left(\frac{W}{\mu_{\mathbb{T}}(\Omega)} \right)<\frac{\alpha}{\mu_{\mathbb{T}}(\Omega)}.\]
Finally we obtain,
\[z^{*}(w)=\iint_\Omega z^{*}(\phi(s,t))\Delta t\Gamma s\leq \iint_{\Omega}\frac{\alpha}{\mu_{\mathbb{T}}(\Omega)}\Delta t\Gamma s=\frac{\alpha}{\mu_{\mathbb{T}}(\Omega)}\cdot\mu_{\mathbb{T}}(\Omega)=\alpha\]
which is a contradiction.
\end{proof}
In the proof of the main theorem, we make use of the following
fixed point theorem of Kubiaczyk.
\begin{theorem}\emph{\cite{kubi2}}\label{fpt} Let $X$ be a metrizable, locally convex
topological vector space, $D$ be a closed convex subset of $X$,
and  $F$ be a weakly sequentially continuous map from $D$ into
itself. If for some $x\in D$ the implication
\begin{equation}\label{fptcond}
\overline{V}=\overline{\text{conv}}(\{x\}\cup F(V)) \Rightarrow V
\ \text{is relatively weakly compact,}
\end{equation}
holds for every subset $V$ of $D$,  then $F$ has a fixed point.
\end{theorem}

\section{The Existence Result}\label{main-res} We claim that in the case of weakly-weakly continuous $f$,
finding a weak solution of \eqref{pde} is equivalent to
solving the integral equation \begin{equation}
\label{pdeinteq}z(x,y)=(C_w)\int_0^x\int_0^y f(s,t,z(s,t))\Delta
t\Gamma s, \ \ (s,t)\in
\mathbb{T}_1\times\mathbb{T}_2.\end{equation} To justify the
equivalence, we first assume that a weakly continuous function $z:
\mathbb{T}_1\times\mathbb{T}_2\to E$ is a weak solution of \eqref{pde}. We show that $z$ solves the integral equation
\eqref{pdeinteq}. By the definition of weak Cauchy integral
(Definition \ref{weakcauchy}), we have
\begin{eqnarray*} (C_w)\int_0^y f(x,t,z(x,t))\Delta t=(C_w)\int_0^y
z^{\Gamma\Delta}(x,t)\Delta
t=z^\Gamma(x,y)-z^\Gamma(x,0)=z^\Gamma(x,y)
\end{eqnarray*}
Note that $z^\Gamma(x,0)=0$ since $z(x,0)=0$. If we integrate the
resulting equation on $[0, x]_{\mathbb{T}_1}$, we obtain
\begin{eqnarray*} (C_w)\int_0^x\int_0^y f(s,t,z(s,t))\Delta t\Gamma
s=(C_w)\int_0^x z^{\Gamma}(s,y)\Gamma s=z(x,y)-z(0,y)=z(x,y)
\end{eqnarray*}
which points out that $z$ solves the integral equation
\eqref{pdeinteq}.

Conversely, we assume that $z(x,y)$ is a solution of the integral
equation \eqref{pdeinteq}. For any $z^*\in E^*$, we have
$$(z^*z)(x,y)=z^*\left(\int_0^x\int_0^y f(s,t,z(s,t))\Delta
t\Gamma s\right)$$ and therefore
\begin{eqnarray*}
(z^*z)^\Gamma(x,y)&=&\left(\int_0^x\int_0^y
z^*(f(s,t,z(s,t)))\Delta t\Gamma
s\right)^\Gamma\\
&=&\int_0^y z^*(f(x,t,z(x,t)))\Delta t.
\end{eqnarray*}
Differentiating the last expression we get
\begin{eqnarray*}
(z^*z)^{\Gamma\Delta}(x,y)&=&\left(\int_0^y z^*(f(x,t,z(x,t)))\Delta t\right)^\Delta\\
&=& z^*(f(x,y,z(x,y))).
\end{eqnarray*} By the definition of weak partial derivatives (Definition
\ref{weakpartialderivative}), we obtain $$z^{\Gamma_w\Delta_w}(x,y)=f(x,y,z(x,y)).$$
Clearly the boundary conditions of \eqref{pde} hold. Hence
$z(x,y)$ is the weak solution of \eqref{pde}.

We consider the space of continuous functions
$\mathbb{T}_1\times\mathbb{T}_2\to E$ with its weak topology,
\emph{i.e.}, $$(C(\mathbb{T}_1\times\mathbb{T}_2, E),
w)=\left(C(\mathbb{T}_1\times\mathbb{T}_2,
E),\tau(C(\mathbb{T}_1\times\mathbb{T}_2,
E),C^*(\mathbb{T}_1\times\mathbb{T}_2, E))\right).$$

Let $G:\mathbb{T}_1\times\mathbb{T}_2\times [0, \infty)\to [0,
\infty)$ be a continuous function and nondecreasing in the last
variable. Assume that the scalar integral inequality
\begin{eqnarray}g(x,y)\geq \int_0^x\int_0^y G(s,t,g(s,t))\Delta t\Gamma s\label{integralineq}\end{eqnarray}
has locally bounded solution $g_0(x,y)$ existing on
$\mathbb{T}_1\times \mathbb{T}_2$.

We define the ball $B_{g_0}$ as follows: \begin{eqnarray}
B_{g_0}=&&\hspace{-0,5 cm}\big{\{}z\in
(C(\mathbb{T}_1\times\mathbb{T}_2, E), w): ||z(x,y)||\leq g_0(x,y)
\text{ on } \mathbb{T}_1\times\mathbb{T}_2,
\nonumber\\&&\hspace{-1 cm}||z(x_1,y_1)-z(x_2,y_2)||\leq
\left|\int_0^{x_2}\int_{y_1}^{y_2}G(s,t,g_0(s,t))\Delta t\Gamma
s\right|\nonumber\\&&\hspace{-0,8
cm}+\left|\int_{x_1}^{x_2}\int_0^{y_1}G(s,t,g_0(s,t))\Delta
t\Gamma s\right| \ \text{ for } x_1, x_2\in \mathbb{T}_1 \ \text{
and } y_1, y_2\in\mathbb{T}_2\big{\}}
\end{eqnarray}
Clearly the set $B_{g_0}$ is nonempty, closed, bounded, convex and
equicontinuous.

Assume that a nonnegative, real-valued, continuous function
$(x,y,r)\mapsto h(x,y,r)$ defined on
$\mathbb{T}_1\times\mathbb{T}_2\times\mathbb{R}^+$ satisfies the
following conditions:
\begin{itemize} \item[(H1)] $h(x,y,0)=0$, \item[(H2)]
$z(x,y)\equiv 0 $ is the unique continuous solution of the
integral inequality
$$u(x,y)\leq\int_0^x\int_0^y h(s,t, u(s,t))\Delta t\Gamma s$$
satisfying the condition $u(0,0)$=0.
\end{itemize} We define the integral operator $F:
(C(\mathbb{T}_1\times\mathbb{T}_2, E), w)\to
(C(\mathbb{T}_1\times\mathbb{T}_2, E), w)$ associated to the integral equation \eqref{pdeinteq} by
\begin{equation}\label{intop} F(z)(x,y)=(\mathcal{R}_w)\int_0^x\int_0^y f(s,t,z(s,t))\Delta
t\Gamma s, \ \ x\in\mathbb{T}_1, \ \ y\in\mathbb{T}_2.
\end{equation}
By the considerations presented above, the fixed point of the
integral operator $F$ is the weak solution of \eqref{pde}. Our
main result is as follows: \begin{theorem} Assume that the
function $f: \mathbb{T}_1\times\mathbb{T}_2\times B_{g_0}\to E$
satisfy the following conditions:
\begin{itemize}
    \item[\emph{(C1)}] $f(x,y,\cdot)$ is weakly-weakly
    sequentially continuous for each $(x,y)\in
    \mathbb{T}_1\times\mathbb{T}_2$,
    \item[\emph{(C2)}] For each strongly absolutely continuous
    function $z:\mathbb{T}_1\times\mathbb{T}_2\to E$, $f(\cdot, \cdot,
    z(\cdot,\cdot))$is weakly continuous
    \item[\emph{(C3)}] $||f(x,y,u)||\leq G(x,y,||u||)$ for all $(x,y)\in
    \mathbb{T}_1\times\mathbb{T}_2$ and $u\in E$,
    \item[\emph{(C4)}] For any function $h$ satisfying the conditions (H1) and (H2) $$\beta(f(I_x\times I_y\times W))\leq
    h(x,y,\beta(W))$$for each $W\subset B_{g_0}$ and $I_x\subset
    \mathbb{T}_1, \ I_y\subset \mathbb{T}_2$.
\end{itemize}
Then there exists a weak solution of the partial dynamic problem \eqref{pde}.
\end{theorem}
\begin{proof}
By virtue of the condition (C2), the operator $F: B_{g_0}\to
(C(\mathbb{T}_1\times\mathbb{T}_2, E), w)$ is well-defined.  Next
we clarify that the operator $F$ maps $B_{g_0}$ into $B_{g_0}$.
For this purpose first we verify $||F(z)(x,y)||\leq g_0(x,y)$. For
$z(x,y)\in B_{g_0}$, the condition (C3), the monotonicity of $G$
in the last variable and the existence of locally bounded solution
$g_0(x,y)$ of \eqref{integralineq} guarantee that
\begin{eqnarray}
||F(z)(x,y)||&=&\left|\left|\int_0^x\int_0^y f(s,t,z(s,t))\Delta t\Gamma s\right|\right|\nonumber\\
&\leq& \int_0^x\int_0^y ||f(s,t,z(s,t))||\Delta t\Gamma s\nonumber\\
&\leq& \int_0^x\int_0^y G(s,t,||z(s,t)||)\Delta t\Gamma s\nonumber\\
&\leq& \int_0^x\int_0^y G(s,t,||g_0(x,y)||)\Delta t\Gamma s\leq
g_0(x,y).
\end{eqnarray}

Consequently, we claim that
\begin{eqnarray*}
||F(z)(x_1,y_1)-F(z)(x_2,y_2)||&\leq&
\left|\int_0^{x_2}\int_{y_1}^{y_2}G(s,t,g_0(s,t))\Delta t\Gamma s
\right|\\&&\hspace{1
cm}+\left|\int_{x_1}^{x_2}\int_0^{y_1}G(s,t,g_0(s,t))\Delta
t\Gamma s\right|.
\end{eqnarray*}
For all $z^*\in E^*$ with $||z^*||\leq 1$, we have
\begin{eqnarray*}|z^*(f(s,t,z(s,t)))|&\leq&\sup\limits_{z^*\in E^*, ||z^*||\leq 1}|z^*(f(s,t,z(s,t)))|\\&=&||(f(s,t,z(s,t)))||\\&\leq&
G(s,t,||z(s,t)||),\end{eqnarray*}where we use the condition (C3)
for the last step. Hence
\begin{eqnarray*}\left|z^*\left[F(z)(x_1,y_1)-F(z)(x_2,y_2)\right]\right|&&\hspace{-0,5
cm}=\left|z^*\left(\int_0^{x_2}\int_{y_1}^{y_2}f(s,t,z)\Delta
t\Gamma s - \int_{x_1}^{x_2}\int_0^{y_1}f(s,t,z)\Delta t\Gamma
s\right)\right|\\&&\hspace{-2 cm}\leq
\int_0^{x_2}\int_{y_1}^{y_2}|z^*(f(s,t,z))|\Delta t\Gamma
s+\int_{x_1}^{x_2}\int_{0}^{y_1}|z^*(f(s,t,z))|\Delta t\Gamma s
\end{eqnarray*}
Utilizing the condition (C2) we acquire, \begin{eqnarray*}
||F(z)(x_1,y_1)-F(z)(x_2,y_2)||&\leq&
\left|\int_0^{x_2}\int_{y_1}^{y_2}G(s,t,||z(s,t)||)\Delta t\Gamma
s \right|\\&&\hspace{1
cm}+\left|\int_{x_1}^{x_2}\int_0^{y_1}G(s,t,||z(s,t)||)\Delta
t\Gamma s\right|.
\end{eqnarray*}
Since $G$ is nondecreasing in the last variable, the desired
result
\begin{eqnarray*} ||F(z)(x_1,y_1)-F(z)(x_2,y_2)||&\leq&
\left|\int_0^{x_2}\int_{y_1}^{y_2}G(s,t,g_0(s,t))\Delta t\Gamma s
\right|\\&&\hspace{1
cm}+\left|\int_{x_1}^{x_2}\int_0^{y_1}G(s,t,g_0(s,t))\Delta
t\Gamma s\right|
\end{eqnarray*} follows.\\
Next, we substantiate the weakly sequentially continuity of the
integral operator $F$. Let $z_{n}\stackrel{w}{\rightarrow} z$ in
$B_{g_0}$. Then for given $\epsilon >0$ there exists
$N\in\mathbb{N}$ such that for any $n>N$ and $(x,y)\in
I_\alpha\times I_\beta\subset \mathbb{T}_1\times \mathbb{T}_2$, we
have $|z^*z_n(x,y)-z^*z(x,y)|<\epsilon.$ Apparently, from
condition (C1), one can obtain
$$|z^*f(x,y,z_n(x,y))-z^*f(x,y,z(x,y))|\leq\frac{\epsilon}{\alpha\beta}.$$
Hence \begin{eqnarray*} \left|z^*
(F(z_n)(x,y)-F(z)(x,y))\right|&&\hspace{-0,5
cm}=\left|z^*\left(\int_0^{x}\int_{0}^{y}f(s,t,z_n)\Delta t\Gamma
s - \int_{0}^{x}\int_0^{y}f(s,t,z)\Delta t\Gamma
s\right)\right|\\&&\hspace{-0,5
cm}\leq\int_0^{x}\int_{0}^{y}\left|z^*\left(f(s,t,z_n(s,t))-f(s,t,z(s,t))\right)\right|\Delta
t\Gamma s\\&&\hspace{-0,5
cm}\leq\int_0^{\alpha}\int_{0}^{\beta}\left|z^*f(s,t,z_n(s,t))-z^*f(s,t,z(s,t))\right|\Delta
t\Gamma s\\&&\hspace{-0,5
cm}<\int_0^{\alpha}\int_{0}^{\beta}\frac{\epsilon}{\alpha\beta}\Delta
t\Gamma s=\epsilon,
\end{eqnarray*}
(for the first integral inequality see
\cite{Guseinov,Book1,Book2}). Owing to the closedness  of
$\mathbb{T}_1\times \mathbb{T}_2$, is it locally compact Hausdorff
space. Thanks to the result of Dobrakov (see \cite{dobrakov}, Thm
9), $F(z_n)$ converges weakly to $F(z)$ in
$(C(\mathbb{T}_1\times\mathbb{T}_2, E), w)$. Therewith $F$ is
weakly sequentially continuous mapping.

As a result $F$ is
well-defined, weakly sequentially continuous and maps $B_{g_0}$
into $B_{g_0}$.

Now we prove that the fixed point of the integral operator
\eqref{intop} exists by employing Kubiaczyk's fixed point theorem
(Theorem \ref{fpt}).

Let $W\subset B_{g_0}$ satisfying the condition
\begin{equation}\label{thmcond}
W=\overline{\text{conv}}\left(\{z\}\cup F(W)\right)
\end{equation} for some
$z\in B_{g_0}$. We prove that $W$ is relatively weakly compact.
For $(x,y)\in \mathbb{T}_1\times \mathbb{T}_2$, we define
$W(x,y)=\{w(x,y)\in E: w\in W\}$. Resulting from Ambrosetti's
Lemma (Lemma \ref{ambrosetti}), the function $(x,y)\mapsto
w(x,y)=\beta(W(x,y))$ is continuous on $\mathbb{T}_1\times
\mathbb{T}_2$.

Since the integral is $\displaystyle{\int_0^x\int_0^y G(s,t,g(s,t))\Delta t\Gamma s}$ is bounded, there exist $\xi\in \mathbb{T}_1$ and $\eta\in
\mathbb{T}_2$ such that $$\iint\limits_R G(s,t,||z(s,t)||)\Delta
t\Gamma s<\epsilon$$ where $R=\mathbb{T}_1\times
\mathbb{T}_2-([0,\xi]_{\mathbb{T}_1}\times
[0,\eta]_{\mathbb{T}_2})$. We divide the interval
$[0,\xi]_{\mathbb{T}_1}$ into $m$ parts
$$0<s_1<s_2<\ldots<s_m=\xi$$ and $[0,\eta]_{\mathbb{T}_2}$ into $n$
parts $$0<t_1<t_2<\ldots<t_m=\eta$$ in a way that each partition
is finer than $\delta>0$. Also we define
$\mathbb{T}_1^i=[s_i,s_{i+1}]_{\mathbb{T}_1}$ and
$\mathbb{T}_2^j=[t_j,t_{j+1}]_{\mathbb{T}_2}$. By Abrosetti's
Lemma there exists $(\sigma_i, \tau_j)\in
\mathbb{T}_1^i\times\mathbb{T}_2^j=P_{ij}$ such that
$$\beta(W(P_{ij}))=\sup\{\beta(W(s,t)): (s,t)\in P_{ij}\}=w(\sigma_i, \tau_j).$$
On the other hand, for $x>\xi$, $y>\eta$ and for any $w\in W$, we
have\begin{eqnarray*}
F(w)(x,y)&=&\int_0^x\int_0^y f(s,t,w(s,t))\Delta t\Gamma s\\
&=&\int_0^\xi\int_0^\eta f(s,t,w(s,t))\Delta t\Gamma s+\iint_{R_1}
f(s,t,w(s,t))\Delta t\Gamma s.
\end{eqnarray*}
Therefore the mean value theorem (Theorem \ref{mvtdouble}) entails
\begin{eqnarray*}
F(w(x,y))\in\sum_{i=0}^{m-1}\sum_{j=0}^{n-1} \mu_\mathbb{T}(P_{ij})
\overline{\text{conv}}(f(P_{ij}\times W(P_{ij})))+\iint_{R_1}
f(s,t,w(s,t))\Delta t\Gamma s,
\end{eqnarray*}
which has the consequence \begin{eqnarray*}
F(W(x,y))\subset\sum_{i=0}^{m-1}\sum_{j=0}^{n-1}
\mu_\mathbb{T}(P_{ij}) \overline{\text{conv}}(f(P_{ij}\times
W(P_{ij})))+\iint_{R_1} f(s,t,W(s,t))\Delta t\Gamma s.
\end{eqnarray*}
Using (C4), Lemma \ref{E_1} and the properties of measure of weak
noncompactness, we acquire
\begin{eqnarray*}
\beta(F(W(x,y)))\hspace{-0,5cm}&&\leq\sum_{i=0}^{m-1}\sum_{j=0}^{n-1}
\mu_\mathbb{T}(P_{ij}) \beta(\overline{\text{conv}}(f(P_{ij}\times
W(P_{ij}))))+\left|\left|\iint_{R_1} f(s,t,w(s,t))\Delta t\Gamma
s\right|\right|\\
\hspace{-0,5cm}&&\leq\sum_{i=0}^{m-1}\sum_{j=0}^{n-1}
\mu_\mathbb{T}(P_{ij}) \beta(f(P_{ij}\times W(P_{ij})))+\sup_{w\in
W}\iint_{R_1} f(s,t,w(s,t))\Delta t\Gamma s\\
\hspace{-0,5cm}&&\leq\sum_{i=0}^{m-1}\sum_{j=0}^{n-1}
\mu_\mathbb{T}(P_{ij}) h((P_{ij}\times \beta(W(P_{ij}))))+\sup_{w\in
W}\iint_{R_1} f(s,t,w(s,t))\Delta t\Gamma s\\
\hspace{-0,5cm}&&\leq\sum_{i=0}^{m-1}\sum_{j=0}^{n-1}
\mu_\mathbb{T}(P_{ij}) h((P_{ij}\times \beta(W(P_{ij}))))+\sup_{w\in
W}\iint_{R} f(s,t,w(s,t))\Delta t\Gamma s\\
\hspace{-0,5cm}&&\leq\sum_{i=0}^{m-1}\sum_{j=0}^{n-1}
\mu_\mathbb{T}(P_{ij}) h((P_{ij}\times w(\sigma_i,\tau_j))+\sup_{w\in
W}\iint_{R} f(s,t,w(s,t))\Delta t\Gamma u\\
\hspace{-0,5cm}&&\leq\sum_{i=0}^{m-1}\sum_{j=0}^{n-1}
\mu_\mathbb{T}(P_{ij}) h((P_{ij}\times w(\sigma_i,\tau_j))+\sup_{w\in
W}\iint_{R} G(s,t,||w(s,t)||)\Delta t\Gamma s\\
\hspace{-0,5cm}&&\leq\sum_{i=0}^{m-1}\sum_{j=0}^{n-1}
\mu_\mathbb{T}(P_{ij}) h((P_{ij}\times w(\sigma_i,\tau_j))+\epsilon
\end{eqnarray*} Since $\epsilon$ is arbitrary,
\begin{equation}\label{lastineq}
\beta(F(W)(x,y))\leq\int_0^x\int_0^y h(s,t,w(s,t))\Delta t\Gamma
s.
\end{equation} By the condition \eqref{thmcond}, inequality \eqref{lastineq} and
the properties of measure of weak noncompactness
$$w(x,y)\leq\int_0^x\int_0^y h(s,t,w(s,t))\Delta t\Gamma s.$$
The condition (H2) implies that the integral inequality above has
only trivial solution, i.e. $w(x,y)=\beta(W(x,y))=0$ which means
that $W$ ie relatively weakly compact. Thus the condition
\eqref{fptcond} of Theorem \ref{fpt} is substantiated. So the
integral operator $F$ defined by \eqref{intop} has a fixed point
which is actually a weak solution of the hyperbolic partial
dynamic equation \eqref{pde}.
\end{proof}
\subsection*{Acknowledgements} The authors wish to acknowledge the anonymous reviewer for his/her detailed and helpful
comments to the manuscript.

\end{document}